\documentclass[10pt,reqno]{amsart}
\usepackage{amsmath,amssymb,latexsym,soul,cite,mathrsfs}
\usepackage{color,enumitem,graphicx}
\usepackage[colorlinks=true,urlcolor=blue,
citecolor=red,linkcolor=blue,linktocpage,pdfpagelabels,
bookmarksnumbered,bookmarksopen]{hyperref}
\usepackage[english]{babel}
\usepackage{enumitem}
\usepackage[left=2.5cm,right=2.5cm,top=3cm,bottom=3cm]{geometry}

\usepackage[hyperpageref]{backref}
\makeatletter
\newcommand{\leqnomode}{\tagsleft@true}
\newcommand{\reqnomode}{\tagsleft@false}
\makeatother

\numberwithin{equation}{section}

\newtheorem{thm}{Theorem}[section]
\newtheorem{lem}[thm]{Lemma}

\newtheorem{Prop}[thm]{Proposition}

\newcommand{\R}{\mathbb{R}}

\title[Nondegeneracy of solutions for a critical Hartree type equation]{Nondegeneracy of solutions for a critical Hartree equation}

\author[J. Giacomoni]{Jacques Giacomoni}
\address[J. Giacomoni]{LMAP, UMR E2S-UPPA CNRS 5142, B\^atiment IPRA, Avenue de l'Universit\'e F-64013 Pau, France}
\email{\href{mailto:email}{jacques.giacomoni@univ-pau.fr}}

\author[Y. Wei]{Yuanhong Wei}
\address[Yuanhong Wei]{Department of mathematics, Jilin University,
	\newline\indent
Changchun, Jilin, P.R. China}
\email{\href{mailto:email}{yhwei@amss.ac.cn}}

\author[Minbo Yang]{Minbo Yang}
\address[ Minbo Yang]{Department of mathematics, Zhejiang Normal University,
	\newline\indent
Jinhua, Zhejiang, P.R. China}
\email{\href{mailto:email}{mbyang@zjnu.edu.cn}}

\thanks{Corresponding author: Minbo Yang}
\thanks{Yuanhong Wei was partially supported by NSFC (11871242).}
\thanks{Minbo Yang was partially supported by NSFC (11571317, 11971634).}


\begin{document}

\begin{abstract}
The aim of this paper is to prove the nondegeneracy of the unique positive solutions for the following critical Hartree type equations when $\mu>0$ is close to $0$,
$$
-\Delta u=\left(I_{\mu}\ast u^{2^{\ast}_{\mu}}\right)u^{{2}^{\ast}_{\mu}-1},~~x\in\mathbb{R}^{N},
$$
where $
I_{\mu}(x)=\frac{\Gamma(\frac{\mu}{2})}{\Gamma(\frac{{N-\mu}}{2})\pi^{\frac{N}{2}}2^{{N-\mu}}|x|^{\mu}}
$ is the Riesz potential and $2^{\ast}_{\mu}=\frac{2{N-\mu}}{N-2}$ is the upper critical exponent due to the Hardy-Littlewood-Sobolev inequality.
\end{abstract}

\maketitle
	

\section{Introduction}	

We are interested in proving the nondegeneracy of the unique positive solutions of the critical Hartree type equation
\begin{equation}\label{cc}
-\Delta u=\left(I_{\mu}\ast u^{2^{\ast}_{\mu}}\right)u^{{2}^{\ast}_{\mu}-1},~~x\in\mathbb{R}^{N},
\end{equation}
where $\mu>0$ is close to $0$, $2^{\ast}_{\mu}=\frac{2{N-\mu}}{N-2}$, $I_{\mu}$ is the Riesz potential defined by
$$
I_{\mu}(x)=\frac{\Gamma(\frac{\mu}{2})}{\Gamma(\frac{{N-\mu}}{2})\pi^{\frac{N}{2}}2^{{N-\mu}}|x|^{\mu}}
$$
with $\Gamma(s)=\displaystyle\int^{+\infty}_{0}x^{s-1}e^{-x}dx$, $s>0$ \Big(In some references the Riesz potential is defined by $
I_{\alpha}(x)=\frac{\Gamma(\frac{N-\alpha}{2})}{\Gamma(\frac{\alpha}{2})\pi^{\frac{N}{2}}2^{\alpha}|x|^{N-\alpha}}
$, $0<\alpha<N$\Big).

 To understand the critical growth for equation \eqref{cc}, we need to recall the well-known Hardy-Littlewood-Sobolev (HLS for short) inequality, see \cite{LE1, LL}.
\begin{Prop}\label{HLS}
 Let $t,r>1$ and $0<\mu<N$ satisfying $1/t+1/r+\mu/N=2$, $f\in L^{t}(\mathbb{R}^N)$ and $h\in L^{r}(\mathbb{R}^N)$. There exists a sharp constant $C(N,\mu,t,r)$, independent of $f,h$, such that
\begin{equation}\label{HLS1}
\int_{\mathbb{R}^{N}}\int_{\mathbb{R}^{N}}\frac{f(x)h(y)}{|x-y|^{\mu}}dxdy\leq C(N,\mu,t,r) |f|_{t}|h|_{r}.
\end{equation}
If $t=r=2N/(2{N-\mu})$, then
$$
C(N,\mu,t,r)=C(N,\mu)=\pi^{\frac{\mu}{2}}\frac{\Gamma(\frac{{N-\mu}}{2})}{\Gamma(N-\frac{\mu}{2})}\left\{\frac{\Gamma(N)}{\Gamma(\frac{N}{2})}\right\}^{\frac{{N-\mu}}{N}}
$$
and there is equality in \eqref{HLS1} if and only if $f\equiv(const.)h$ and
$$
h(x)=A(\gamma^{2}+|x-a|^{2})^{-(2{N-\mu})/2}
$$
for some $A\in \mathbb{C}$, $0\neq\gamma\in\mathbb{R}$ and $a\in \mathbb{R}^{N}$.
\end{Prop}
On one hand, related to the study of the HLS inequality, Lieb \cite{LE1} classified all the maximizers of the HLS functional under constraints and obtained the best constant, then he posed the classification of the positive solutions of
\begin{equation}\label{Int}
u(x)=\int_{\mathbb{R}^{N}}\frac{u(y)^{\frac{N+\alpha}{N-\alpha}}}{|x-y|^{N-\alpha}}dy,~~x\in\mathbb{R}^{N},
\end{equation}
as an open problem. Chen, Li and Ou \cite{CLO1} developed the method of moving planes in integral forms to prove that any critical points of the functional was radially symmetric and assumed the unique form and gave a positive answer to Lieb's open problem.
In fact, equation \eqref{Int} is an equivalent form for the fractional equation
\begin{equation}\label{Frac}
(-\Delta)^{\frac{\alpha}{2}}u=u^{\frac{N+\alpha}{N-\alpha}},~~x\in\mathbb{R}^{N}.
\end{equation}
When $N\geq3$, $\alpha=2$ and $s=0$, equation \eqref{Frac} goes back to
\begin{equation}\label{lcritical}
-\Delta u=u^{\frac{N+2}{N-2}},~~x\in\mathbb{R}^{N}.
\end{equation}
Equation \eqref{lcritical} is in fact the Euler-Lagrange equation
of the extremal functions of the Sobolev inequality.  Caffarelli, Gidas and Spruck \cite{CGS} proved the symmetry and uniqueness of the positive solutions respectively. Chen and Li \cite{CL}, Li \cite{LC} simplified the results above as an application of the moving plane method. Li \cite{Li} used moving sphere method. The classification of the solutions of equation \eqref{lcritical} plays an important role in the Yamabe problem, the prescribed scalar curvature problem on Riemannian manifolds and the priori estimates in nonlinear equations. It is well known that, Aubin \cite{A}, Talenti \cite{Ta} proved that the best Sobolev constant $S$
 can be achieved by a two-parameter solutions of the form
\begin{equation}\label{U0}
 U_0(x):=[N(N-2)]^{\frac{N-2}{4}}\Big(\frac{t}{t^2+|x-\xi|^{2}}\Big)^{\frac{N-2}{2}}.
 \end{equation}
Furthermore, equation \eqref{lcritical}
has an $(N+1)$-dimensional manifold of solutions given by
$$
\mathcal{Z}=\left\{z_{t,\xi}(x)=[N(N-2)]^{\frac{N-2}{4}}\Big(\frac{t}{t^2+|x-\xi|^{2}}\Big)^{\frac{N-2}{2}},
\xi\in\mathbb{R}^{N}, t\in\mathbb{R}^{+}\right\}.
$$
It was proved in \cite{R} that $Z\in \mathcal{Z}$ is said to be nondegenerate in the sense that the linearized equation
\begin{equation}\label{Linearized}
-\Delta v=Z^{\frac{4}{N-2}}v,~~x\in\mathbb{R}^{N},
\end{equation}
in $D^{1,2}(\mathbb{R}^N)$ only admits solutions of the form
$$
\eta=aD_{t}Z+\mathbf{b}\cdot\nabla Z,
$$
where $a\in\mathbb{R},\mathbf{b}\in\mathbb{R}^{N}$.

On the other hand, if we consider the combination of the HLS inequality and the Sobolev inequality, for every $u\in H^{1}(\mathbb{R}^{N})$, the integral
$$
\int_{\mathbb{R}^{N}}\int_{\mathbb{R}^{N}}\frac{|u(x)|^{p}|v(y)|^{p}}{|x-y|^{\mu}}dxdy
$$
is well defined if
$$
\frac{2N-\mu}{N}\leq p\leq\frac{2N-\mu}{N-2}.
$$
So, we call $2_{\mu}^{\ast}:=\frac{2N-\mu}{N-2}$ is the upper critical exponent due to the Hardy--Littlewood--Sobolev inequality.
To study the best constant for the critical imbedding, we may study the following minimizing problem
\begin{equation}\label{M}
S_{HL}=\inf_{u\in D^{1,2}(\mathbb{R}^{N})\backslash\{0\}}\frac{\displaystyle\int_{\mathbb{R}^{N}}|\nabla u|^{2}dx}{\left(\displaystyle\int_{\mathbb{R}^{N}}(I_{\mu}\ast |u|^{2^{\ast}_{\mu}})|u|^{2^{\ast}_{\mu}}dy\right)^{\frac{1}{2^{\ast}_{\mu}}}}.
\end{equation}
Obviously, equation \eqref{cc} is the Euler-Lagrange equation for minimizing problem \eqref{M}.

There are also some uniqueness and nondegeneracy results for the subcritical Choquard equation
$$
-\Delta u+u=\left(I_{\mu}\ast u^{p}\right)u^{p-1},~~x\in\mathbb{R}^{N}.
$$
In \cite{Len, WW} the authors proved the uniqueness and nondegeneracy for the case $N=3, \mu=1$ and $p=2$. Xiang \cite{X} generalized the results a little by showing the nondegeneracy when $p>2$ is close $2$. In \cite{SJ}, J. Seok proved the limit profile of the ground states and he also proved the uniqueness and nondegeneracy results if $\mu$ close to $0$ or $N$. The methods in \cite{SJ} depends a lot on the functional analysis techniques and the embedding properties for $H^{1}(\R^N)$. However, to study the nondegeneracy of the unique positive solutions for the critical equation \eqref{cc} is not so easy, because we have only continuous embedding from $D^{1,2}(\R^N)$ into $L^{2^*}(\R^N)$. For the study of Choquard equation, we may refer the authors turn to \cite{MS1}.

Notice that, by the convergence property of the Riesz potential that $I_{\mu}\to \delta_x$ as $\mu\to N$,  see \cite{Lan}, we also find that equation \eqref{cc} goes to equation \eqref{lcritical} as $\mu\to N$. Thus equation \eqref{lcritical} can also be treated as the limit equation \eqref{cc}. In this sense, we may consider equation \eqref{cc} as a generalization of equation \eqref{lcritical} from a nonlocal point of view. The uniqueness of the positive solutions are recently proved by \cite{DY, Lei, Liu} separately by using moving plane methods.
Recently, using the classification results in \cite{DY}, we can also find some strong converging property for the ground states when the parameter $\mu$ approaches the $0$ or $N$. The nondegeneracy of the ground states for the critical Hartree equation plays an important role in studying semiclassical problems for the critical Choquard equation. As far as we know, the first nondegeneracy result for the critical Hartree equation is due to Du and Yang in \cite{DY}, where the authors proved the  nondegeneracy of $U_\mu$ as $\mu$ close to $N$.
\begin{lem}\label{TH3} \cite{DY}
Assume that $0<\mu<N$, $N=3\ \hbox{or} \ 4$. If $\mu$ is close to $N$ then the linearized equation at the unique positive solution $U_\mu$ \begin{equation}\label{Linearized E}
-\Delta\psi-2^{\ast}_{\mu}(I_{\mu}\ast(U_\mu^{2^{\ast}_{\mu}-1}\psi))U_\mu^{{2}^{\ast}_{\mu}-1}-(2^{\ast}_{\mu}-1)(I_{\mu}\ast U_\mu^{2^{\ast}_{\mu}})U_\mu^{{2}^{\ast}_{\mu}-2}\psi=0,~~x\in\mathbb{R}^{N},
\end{equation}
only admits solutions in $D^{1,2}(\mathbb{R}^{N})$ of the form
$$
\psi=aD_{t}{U}_\mu+\mathbf{b}\cdot\nabla {U}_\mu,
$$
where $a\in\mathbb{R},\mathbf{b}\in\mathbb{R}^{N}$.
\end{lem}

In the present paper, we are going to prove the nondegeneracy of the unique positive solutions for the critical Hartree equations with $\mu>0$ close to $0$. Notice that the function $I_{\mu}$ blows up when $\mu\rightarrow0$ due to the fact that the term $\Gamma(\frac{\mu}{2})$ in the coefficient of $I_{\mu}$. And so, to get rid of this singular term, we take a scaling by
\begin{equation}\label{scaled Solutions}
	V_{\mu}=S(N,\mu)U_{\mu},
\end{equation}
where
\[
	 S(N,\mu):=\Big(\frac{\Gamma(\frac{\mu}{2})}{\Gamma(\frac{{N-\mu}}{2})\pi^{\frac{N}{2}2^{N-\mu}}}\Big)^{\frac{1}{{2}^{\ast}_{\mu}-2}}\sim \Big(\frac{1}{\mu}\Big)^{\frac{1}{{{2}^{\ast}_{\mu}-2}}} \mbox{ as } \mu \rightarrow 0.
\]
Then, to study the the nondegeneracy of $V_\mu$ for \eqref{cc} as $\mu$ close to $0$, one needs only to study the corresponding property of the solutions $V_\mu$ for
\begin{equation}\label{scaled Equa}
	-\Delta v=\Big(\int_{\mathbb{R}^N}\frac{v(y)^{{{2}^{\ast}_{\mu}}}}{|x-y|^{\mu}}dy \Big)v^{{{2}^{\ast}_{\mu}-1}},~~x\in\mathbb{R}^{N}.
\end{equation}

 In the following we recall some basic results about the best constant $S_{HL}$ defined in \eqref{M} and the existence of positive ground state solutions for \eqref{scaled Equa}. We have the following lemma taken from \cite{AGSY, DY, GY}.
\begin{lem}\label{lem1}
The best constant $S_{HL}$ defined in \eqref{M} satisfies
\begin{equation}\label{WSX}
S_{HL}=\frac{S}{C(N,\mu)^{\frac{N-2}{2{N-\mu}}}},
\end{equation}
where $S$ is the Sobolev constant. What's more,
$$
{V}_\mu(x)=S^{\frac{({N-\mu})(2-N)}{4({N-\mu}+2)}}C(N,\mu)^{\frac{2-N}{2({N-\mu}+2)}}U_0
$$
is the unique family of radial functions that achieves $S_{HL}$ and satisfies equation \eqref{scaled Equa}, where $U_0$ is defined in \eqref{U0}.
\end{lem}
This Lemma characterizes the relations among the best constant $S_{HL}$,  the Sobolev constant $S$ and the best constant $C(N,\mu)$ in the HLS inequality.

The main result of the present paper is about the nondegeneracy of $V_\mu$ as $\mu$ close to $0$. That is
\begin{thm}\label{thm: nondegeneracy}
	Let $\mu \in (0,N)$ sufficiently close to $0$ and $V_{\mu}$ be the corresponding family of unique positive radial
	solution of \eqref{scaled Equa}. Then the linearized equation
	of \eqref{scaled Equa} at $V_{\mu}$ , given by
	\begin{equation}\label{Le}
		-\Delta \phi -{{2}^{\ast}_{\mu}}(\frac{1}{|x|^{\mu}}\ast(V_{\mu}^{{{2}^{\ast}_{\mu}}-1}\phi))V_{\mu}^{{{2}^{\ast}_{\mu}}-1}-({{2}^{\ast}_{\mu}}-1)(\frac{1}{|x|^{\mu}}\ast V_{\mu}^{{2}^{\ast}_{\mu}})V_{\mu}^{{{2}^{\ast}_{\mu}}-2}\phi =0
	\end{equation}
only admits solutions in $D^{1,2}(\mathbb{R}^{N})$ of the form
	\[
		\phi=aD_{t}{V}_\mu+\mathbf{b}\cdot\nabla {V}_\mu
	\]
where $a\in\mathbb{R},\mathbf{b}\in\mathbb{R}^{N}$.
\end{thm}

\section{Nondegeneracy for the limit problem}
To prove the main results, we need to study a limit problem for the critical Hartree equation \eqref{scaled Equa} as $\mu\to 0$ first.
\begin{Prop}
The nonlocal equation
	\begin{equation}\label{Lim}
		-\Delta v = \Big(\int_{\mathbb{R}^N} v^{{2}^{\ast}}dx\Big) v^{{2}^{\ast}-1},~~x\in\mathbb{R}^{N},
	\end{equation}
has a family of unique positive solution $V_0$ of the form
$$
V_0=S^{\frac{N(2-N)}{4(N+2)}}U_0,
$$
 where $U_0$ is defined in \eqref{U0}. Moreover, the linearized equation of \eqref{Lim} at $V_0$ given by
	\[
		L_1(\varphi):=-\Delta \varphi -{2}^{\ast}\int_{\mathbb{R}^N}V_0^{{2}^{\ast}-1}\varphi dx V_0^{{2}^{\ast}-1} -({2}^{\ast}-1)\int_{\mathbb{R}^N} V_0^{2^\ast}dxV_0^{{2}^{\ast}-2} \varphi=0 \phantom{==} \mbox{ in }\mathbb{R}^N
	\]
only admits solutions in $D^{1,2}(\mathbb{R}^{N})$ of the form
	\[
		\varphi=aD_{t}{V}_0+\mathbf{b}\cdot\nabla {V}_0
	\]
where $a\in\mathbb{R},\mathbf{b}\in\mathbb{R}^{N}$.
\end{Prop}

\begin{proof}
Let $v_1$ and $v_2$ be two positive radial solutions for \eqref{Lim},
that is
$$
-\Delta v_1 = \int_{\mathbb{R}^N} v_1^{{2}^{\ast}}dx v_1^{{2}^{\ast}-1} \phantom{==} \mbox{ in } \mathbb{R}^N,
$$
and
$$
-\Delta v_2 = \int_{\mathbb{R}^N} v_2^{{2}^{\ast}}dx v_2^{{2}^{\ast}-1} \phantom{==} \mbox{ in } \mathbb{R}^N,
$$
By defining $a_1=\displaystyle\int_{\mathbb{R}^N} v_1^{{2}^{\ast}}dx$ and $a_2=\displaystyle\int_{\mathbb{R}^N} v_2^{{2}^{\ast}}dx$, $v_1, v_2$ are positive radial solutions to
$
-\Delta v_1 = a_1 v_1^{{2}^{\ast}-1},
$
and
$
-\Delta v_2 = a_2 v_2^{{2}^{\ast}-1},
$
in $\mathbb{R}^N$. Notice that $\Big(\frac{a_1}{a_2}\Big)^{\frac{1}{{2}^{\ast}-2}}v_1$ also satisfies the latter equation, by the uniqueness of the positive radial solution for $-\Delta v =a v^{{2}^{\ast}-1} $, we know that
$$
\Big(\frac{a_1}{a_2}\Big)^{\frac{1}{{2}^{\ast}-2}}v_1=v_2.
$$
Consequently, we have $a_1=a_2$ and then $v_1=v_2$ since both $v_1, v_2$ satisfy equation \eqref{Lim}. By the uniqueness of $U_0$, direct computation shows the unique positive solution $V_0$ for \eqref{Lim} is of the form
$$
V_0=S^{\frac{N(2-N)}{4(N+2)}}U_0.
$$

To prove the nondegeneracy, let $a_0=\displaystyle\int_{\mathbb{R}^N} V_0^{{2}^{\ast}}dx$, then $V_0$ satisfies
	\begin{equation}\label{eq:constant lane-emden}
		-\Delta u= a_0 u^{{2}^{\ast}-1}.
	\end{equation}			
The linearized equation of it at $V_0$ is given by
	\[
		L_2(\varphi):=-\Delta \varphi-({2}^{\ast}-1)a_0 V_0^{{2}^{\ast}-2}\varphi=0
	\]
only admits solutions in $D^{1,2}(\mathbb{R}^{N})$ of the form
		\begin{equation}\label{SF}
		\varphi=aD_{t}{V}_0+\mathbf{b}\cdot\nabla {V}_0
\end{equation}
where $a\in\mathbb{R},\mathbf{b}\in\mathbb{R}^{N}$.

If the conclusion is not true, suppose that $L_1$ has a nontrivial solution $\varphi $ in $D^{1,2}(\mathbb{R}^{N})$, which is not of the form in \eqref{SF}. Then we may assume that $\varphi$ is $D^{1,2}(\mathbb{R}^{N})$ orthogonal to $\partial_{x_i}V_0$ for every $i = 1,..., N$ and $D_{t}{V}_0$. By denoting $\lambda := {2}^{\ast}\displaystyle\int_{\mathbb{R}^N}V_0^{{2}^{\ast}-1}\varphi dx$, from
\[
L_1(\varphi)=0,
\]
we know
$L_2(\varphi) = \lambda V_0^{{2}^{\ast}-1}$, and so $\lambda\neq 0$.
Moreover, since
$$
\begin{aligned}
			L_2\left(\frac{\lambda}{(2-{{2}^{\ast}}) a_{0}} V_{0}\right)&=\frac{\lambda}{(2-{{2}^{\ast}) a_{0}}} L_2\left(V_{0}\right)\\
&=\frac{\lambda}{(2-{{2}^{\ast}}) a_{0}}\left(-\Delta V_{0}-({{2}^{\ast}}-1) a_{0} V_{0}^{{{2}^{\ast}}-1}\right)\\
&=\lambda V_{0}^{{{2}^{\ast}}-1},
\end{aligned}
$$
we know $$L_2(\varphi-\frac{\lambda}{(2-{{2}^{\ast}}) a_{0}} V_0)\equiv0.$$
This implies that there are some  $a\in\mathbb{R},\mathbf{b}\in\mathbb{R}^{N}$. such that
\begin{equation}\label{EFU}
		\varphi-\frac{\lambda}{(2-{{2}^{\ast}}) a_{0}} V_{0}=aD_{t}{V}_0+\mathbf{b}\cdot\nabla {V}_0.
	\end{equation}
We prove that  $a=0,\mathbf{b}=(b_1,b_2,...., b_N)=\mathbf{0}$. In fact, taking the inner product of the left part with $\partial_{x_j}V_0$ and integrating, by the fact that $\varphi$ is $D^{1,2}(\mathbb{R}^{N})$ orthogonal to $\partial_{x_i}V_0$ for every $i = 1,..., N$, we know
	\[
\begin{aligned}
		\int_{\mathbb{R}^{N}} \nabla\varphi\nabla\partial_{x_{j}} V_{0} d x-\frac{\lambda}{(2-{{2}^{\ast}}) a_{0}} \int_{\mathbb{R}^{N}} \nabla V_{0} \nabla\partial_{x_{j}} V_{0} d x
&=-\frac{\lambda}{(2-{{2}^{\ast}}) a_{0}} \int_{\mathbb{R}^{N}} \frac{1}{2} \partial_{x_{j}}\left(|\nabla V_{0}|^{2}\right) d x=0.
	\end{aligned}
\]
Similarly, by the fact that $\varphi$ is $D^{1,2}(\mathbb{R}^{N})$ orthogonal to $D_{t}{V}_0$, we also have
\[
\begin{aligned}
		\int_{\mathbb{R}^{N}} \nabla\phi \nabla D_{t}{V}_0 d x-\frac{\lambda}{(2-{{2}^{\ast}}) a_{0}} \int_{\mathbb{R}^{N}} \nabla V_{0} \nabla D_{t}{V}_0d x
&=-\frac{\lambda}{(2-{{2}^{\ast}}) a_{0}} \int_{\mathbb{R}^{N}} \frac{1}{2} D_{t}\left(|\nabla V_{0}|^{2}\right) d x=0,
	\end{aligned}
\]
since $\displaystyle\int_{\mathbb{R}^{N}}|\nabla V_{0}|^{2}d x$ is invariant in $t$.

	On the other hand, taking inner product on the right part of \eqref{EFU} by $\partial_{x_j}V_0$, we get
	\begin{equation}
\begin{aligned}
a\int_{\mathbb{R}^{N}}& \nabla\partial_{x_j}V_0 \nabla D_{t}{V}_0 d x+b_{j} \int_{\mathbb{R}^{N}}|\nabla\partial_{x_{j}} V_{0}|^{2} d x+\underset{i \neq j}{\sum} b_{i} \int_{\mathbb{R}^{N}} \nabla\partial_{x_{i}} V_{0}\nabla\partial_{x_{j}} V_{0} dx \\
&=a\sum_{m=1}^{N} \int_{\mathbb{R}^{N}} \frac{ x_{j}}{r} D_{t}V_{0,m} V_{0,m}^{\prime}(r) dx+b_{j} \int_{\mathbb{R}^{N}}|\nabla\partial_{x_{j}} V_{0}|^{2} dx+\underset{i \neq j}{\sum} \sum_{m=1}^{N}b_{i} \int_{\mathbb{R}^{N}} \frac{x_{i} x_{j}}{r^{2}}|V_{0,m}^{\prime}(r)|^2 dx
\end{aligned}
	\end{equation}
where $V_{0,m}=\partial_{x_{m}} V_{0}$, $1\leq m\leq N$. Since $\frac{ x_{j}}{r} D_{t}V_{0,m}V_{0,m}^{\prime}(r) $ and
	$\frac{x_{i} x_{j}}{r^{2}} |V_{0,m}^{\prime}(r)|^2$ is odd in variables $x_i$ and $x_j$, we know
$$
a\int_{\mathbb{R}^{N}}\nabla\partial_{x_j}V_0 \nabla D_{t}{V}_0 d x+b_{j} \int_{\mathbb{R}^{N}}|\nabla\partial_{x_{j}} V_{0}|^{2} dx+\underset{i \neq j}{\sum} b_{i} \int_{\mathbb{R}^{N}} \nabla\partial_{x_{i}} V_{0}\nabla\partial_{x_{j}} V_{0} dx=b_{j} \int_{\mathbb{R}^{N}}|\nabla\partial_{x_{j}} V_{0}|^{2} dx,
$$
consequently, $b_{j}=0$,  $1\leq j\leq N$, i.e., $\mathbf{b}=\mathbf{0}$. By similar arguments, we know
	\begin{equation}
\begin{aligned}
a&\int_{\mathbb{R}^{N}} |\nabla D_{t}{V}_0|^2dx +\sum_{i=1}^{N}b_{i} \int_{\mathbb{R}^{N}} \nabla D_{t}{V}_0\nabla\partial_{x_{i}}V_{0}dx\\
&=a\int_{\mathbb{R}^{N}} |\nabla D_{t}{V}_0|^2dx +\sum_{m=1}^{N}b_i \int_{\mathbb{R}^{N}}\frac{ x_{i}}{r} D_{t}V_{0,m} V_{0,m}^{\prime}(r)dx\\
&=a\int_{\mathbb{R}^{N}} |\nabla D_{t}{V}_0|^2dx,
\end{aligned}
	\end{equation}
which implies $a=0$.

Combining the above arguments, we conclude that
$$
\varphi-\frac{\lambda}{(2-{{2}^{\ast}}) a_{0}} V_{0}=0,
$$
i.e.
$$
\varphi=\frac{\lambda}{(2-{{2}^{\ast}}) a_{0}} V_{0}.
$$
Now note that
	\[
		\lambda={2}^{\ast} \int_{\mathbb{R}^{N}} V_{0}^{{2}^{\ast}-1} \phi d x={2}^{\ast} \frac{\lambda}{(2-{2}^{\ast}) a_{0}} \int_{\mathbb{R}^{N}} V_{0}^{{2}^{\ast}} d x={2}^{\ast} \frac{\lambda}{2-{2}^{\ast}}.
	\]
This implies that ${2}^{\ast}=1$, which obviously is a contradiction.
\end{proof}

\section{Proof of the main results}
\begin{Prop}\label{Convergence}
	Let $\{V_{\mu}\}$ be the unique family of positive solutions to \eqref{scaled Equa}
 and $V_0\in D^{1,2}(\mathbb{R}^N)$ be a unique positive radial state of \eqref{Lim}. Then one has
	\[
		\lim\limits_{\mu\rightarrow 0} \| V_{\mu}-V_{0} \|_{D^{1,2}(\mathbb{R}^N)}=0.
	\]
\end{Prop}

\begin{proof}
Notice that
	\[
		V_{\mu}=S^{\frac{(N-\mu)(2-N)}{4(N-\mu+2)}}C(N,\mu)^{\frac{2-N}{2(N-\mu+2)}}U_0,\ \ V_0=S^{\frac{N(2-N)}{4(N+2)}}U_0,
	\]
since $$C(N,\mu)=\pi^{\frac{\mu}{2}}\frac{\Gamma(\frac{{N-\mu}}{2})}{\Gamma(N-\frac{\mu}{2})}\left\{\frac{\Gamma(N)}{\Gamma(\frac{N}{2})}\right\}^{\frac{{N-\mu}}{N}}\to 1,\ \hbox{as}\ \mu\to 0,$$
we know
\[
		\lim\limits_{\mu\rightarrow 0} \| V_{\mu}-V_{0} \|_{D^{1,2}(\mathbb{R}^N)}=0.
	\]
\end{proof}

We introduce the following equivalent form of the HLS inequality with Riesz potential.
\begin{Prop}\label{Prop}
Let $1\leq r<s<\infty$ and $0<\mu<N$ satisfy
$$
\frac{1}{r}-\frac{1}{s}=\frac{N-\mu}{N}.
$$
Then for $\mu$ sufficient close to $0$, there exists constant $C>0$ such that for any $f\in L^{r}(\mathbb{R}^{N})$, there holds
\begin{equation}\label{RFV}
\|\frac{1}{|\cdot|^{\mu}}\ast f\|_{L^{s}(\mathbb{R}^{N})}\leq C\|f\|_{L^{r}(\mathbb{R}^{N})}.
\end{equation}
\end{Prop}
\begin{proof}
This is due to the equivalent form of Hardy-littlewood-Sobolev inequality and the fact that the best constant $C(N, \mu)\to 1$ as $\mu\rightarrow 0$, see \cite{St}.
\end{proof}


 The corresponding Euler-Lagrange functional of \eqref{scaled Equa} is
$$
J_{\mu}(v)=\frac{1}{2}\int_{\mathbb{R}^{N}}|\nabla v|^{2}dx-\frac{1}{2\cdot2^{\ast}_{\mu}}\int_{\mathbb{R}^{N}}(\frac{1}{|x|^{\mu}}\ast v^{2^{\ast}_{\mu}})v^{{2}^{\ast}_{\mu}}dx.
$$
and the derivative of $J_{\mu}(v)$ is
$$
<J^{\prime}_{\mu}(v),\phi>=\int_{\mathbb{R}^{N}}\nabla v\nabla \phi dx-\int_{\mathbb{R}^{N}}(\frac{1}{|x|^{\mu}}\ast v^{2^{\ast}_{\mu}})v^{{2}^{\ast}_{\mu}-1}\phi dx,~~\phi\in D^{1,2}(\mathbb{R}^{N}).
$$
The functional $J_{\mu}(v)$ possesses the family of critical points, depending on $(N+1)$-parameters $\xi\in\mathbb{R}^{N}$ and $t\in\mathbb{R}^{+}$,
$$
z_{t,\xi,\mu}=S^{\frac{(N-\mu)(2-N)}{4(N-\mu+2)}}C^{\ast}(N,\mu)^{\frac{2-N}{2(N-\mu+2)}}[\frac{\sqrt{N(N-2)}t}{(t^{2}+|x-\xi|^{2})}]^{\frac{N-2}{2}}.
$$
Then we define $Z_{\mu}$ is the $(N+1)$-dimensional critical manifold of $J^{\prime}_{\mu}(v)$ which satisfies
$$
Z_{\mu}=\{z_{t,\xi,\mu}|t>0,\xi\in\mathbb{R}^{N}\},
$$
and
$T_{z}Z_{\mu}$ is the tangent space to $Z_{\mu}$. We have
$$
<J^{\prime\prime}_{\mu}(z)[w],\phi>=0,~~\forall w\in T_{z}Z_{\mu},~~\forall\phi\in D^{1,2}(\mathbb{R}^{N}).
$$

If $\mu$ is close to $0$ or $N$, the nondegeneracy of the ground states of the subcritical Choquard equation was studied in \cite{SJ}. Inspired by \cite{SJ}, we are going to prove that $Z_{\mu}$ satisfies the nondegeneracy condition when $\mu$ is close to $0$.

{\bf Proof of Theorem \ref{thm: nondegeneracy}.}
Fix $t, \xi$, since any $V_\mu\in Z_{\mu}$ satisfies \eqref{scaled Equa}, we can conclude that any $\partial_{t}V_\mu=\frac{\partial V_\mu}{\partial t}$, $\partial_{i}V_\mu=\frac{\partial V_\mu}{\partial x_{i}}\in T_{z}{Z_{\mu}}(i=1,2,\cdots, N)$ satisfy the following equation:
\begin{equation}\label{A}
A_{\mu}(\psi)=-\Delta\psi-2^{\ast}_{\mu}(\frac{1}{|x|^{\mu}}\ast(V_\mu^{2^{\ast}_{\mu}-1}\psi))V_\mu^{{2}^{\ast}_{\mu}-1}-(2^{\ast}_{\mu}-1)(\frac{1}{|x|^{\mu}}\ast V_\mu^{2^{\ast}_{\mu}})V_\mu^{{2}^{\ast}_{\mu}-2}\psi=0,~~\psi\in D^{1,2}(\mathbb{R}^{N}).
\end{equation}
It is obvious that $T_{z}{Z_{\mu}}\subseteq\mathbf{Ker}{[J_{\mu}^{\prime\prime}(z)]}$, next we show that $\mathbf{Ker}{[J_{\mu}^{\prime\prime}(z)]}\subseteq T_{z}{Z_{\mu}}$. Noting that $V_\mu=cU_0$ and recalling the finite dimensional vector space
$$
T_{z}{Z_{\mu}}:=\mathbf{span}\{\partial_{1}V_\mu,\partial_{2}V_\mu,\cdots,\partial_{N}V_\mu,\partial_{t}V_\mu\}=\mathbf{span}\{\partial_{1}U_0,\partial_{2}U_0,\cdots,\partial_{N}U_0,\partial_{t}U_0\}.
$$
On the contrary, we suppose that there exists a sequence $\{\mu_{n}\}$ with $\mu_{n}\rightarrow 0$ as $n\rightarrow\infty$ and for each $\mu_{n}$ we have nontrivial solution $\psi_{n}$ of \eqref{A} in the complement of $T_{z}{Z_{\mu}}$ in $L^{2^{\ast}}(\mathbb{R}^{N})$.

We define the operator
$$
L[\psi]=2^{\ast}_{\mu}(\frac{1}{|x|^{\mu}}\ast(V_\mu^{2^{\ast}_{\mu}-1}\psi))V_\mu^{{2}^{\ast}_{\mu}-1}+(2^{\ast}_{\mu}-1)(\frac{1}{|x|^{\mu}}\ast V_\mu^{2^{\ast}_{\mu}})V_\mu^{{2}^{\ast}_{\mu}-2}\psi,
$$
then for any $\varphi\in D^{1,2}(\mathbb{R}^{N})$, H$\ddot{o}$lder's inequality and Proposition \ref{Prop} implies that
\begin{equation}\nonumber
\begin{aligned}
|<L[\psi],\varphi>|
&\leq2^{\ast}_{\mu}|\int_{\mathbb{R}^{N}}(\frac{1}{|x|^{\mu}}\ast(V_\mu^{2^{\ast}_{\mu}-1}\psi))V_\mu^{{2}^{\ast}_{\mu}-1}\varphi dx|+(2^{\ast}_{\mu}-1)|\int_{\mathbb{R}^{N}}(\frac{1}{|x|^{\mu}}\ast V_\mu^{2^{\ast}_{\mu}})V_\mu^{{2}^{\ast}_{\mu}-2}\psi\varphi dx|\\
&\leq C\left(||V_\mu||^{\frac{N+2-\mu}{N}}_{L^{{2}^{\ast}}}+||V_\mu||^{\frac{2N-\mu}{N-2}}_{L^{{2}^{\ast}}}||V_\mu||^{\frac{4-\mu}{(2N-\mu){2}^{\ast}}}_{L^{\frac{4-\mu}{N-2}}}\right)||\psi||_{L^{2^{\ast}}(\mathbb{R}^{N})}||\varphi||_{D^{1,2}}\\
\end{aligned}
\end{equation}
where we used the integrability of $V_\mu\in L^{p}(\mathbb{R}^{N}), (\frac{N}{N-2}<p\leq\infty)$ and $\psi\in L^{2^{\ast}}(\mathbb{R}^{N})$. Therefore we find that the functional $L[\psi]\in(D^{1,2}(\mathbb{R}^{N}))^{\ast}$ where $(D^{1,2}(\mathbb{R}^{N}))^{\ast}$ denote the dual space of $D^{1,2}(\mathbb{R}^{N})$. Since $-\Delta\psi\in(D^{1,2}(\mathbb{R}^{N}))^{\ast}$, then we achieve that $\psi\in D^{1,2}(\mathbb{R}^{N})$.

Now we may assume that $\psi_{n}$ is a sequence of unit solutions for the linearized equation at $V_{n}:=V_{\mu_n}$, hence there exists $\psi_{0}\in D^{1,2}(\mathbb{R}^{N})$, such that $\psi_{n}\rightharpoonup\psi_{0}$ in $D^{1,2}(\mathbb{R}^{N})$ as $n\rightarrow\infty$. Consequently for any $\varphi\in D^{1,2}(\mathbb{R}^{N})$,
\begin{equation}\label{B}
\begin{aligned}
\int_{\mathbb{R}^{N}}\nabla\psi_n\nabla\varphi dx=&2^{\ast}_{\mu_n}\int_{\mathbb{R}^{N}}\Big(\frac{1}{|x|^{\mu_n}}\ast(V^{2^{\ast}_{\mu_{n}}-1}_{n}\psi_n)\Big)V^{2^{\ast}_{\mu_{n}}-1}_{n}\varphi dx\\
&+(2^{\ast}_{\mu_n}-1)\int_{\mathbb{R}^{N}}\Big(\frac{1}{|x|^{\mu_n}}\ast(V^{2^{\ast}_{\mu_{n}}}_n)\Big)V^{2^{\ast}_{\mu_{n}}-2}_{n}\psi_n\varphi dx.
\end{aligned}
\end{equation}

We claim that
\begin{equation}\label{Claim1}
\int_{\mathbb{R}^{N}}\Big(\frac{1}{|x|^{\mu_n}}\ast (V^{2^{\ast}_{\mu_{n}}-1}_{n}\psi_{n})\Big)V^{2^{\ast}_{\mu_{n}}-1}_{n}\varphi dx\rightarrow\Big(\int_{\mathbb{R}^{N}}V^{2^{\ast}-1}_{0}\psi_{0}dx\Big)\Big(\int_{\mathbb{R}^{N}}V^{2^{\ast}-1}_{0}\varphi dx\Big), \ \ \hbox{as} \ \ n\to \infty
\end{equation}
and
\begin{equation}\label{Claim2}
\int_{\mathbb{R}^{N}}\Big(\frac{1}{|x|^{\mu_n}}\ast V^{2^{\ast}_{\mu_{n}}}_n\Big)V^{2^{\ast}_{\mu_{n}}-2}_{n}\psi_n\varphi dx\rightarrow\Big(\int_{\mathbb{R}^{N}}V^{2^{\ast}}_{0}dx\Big)\Big(\int_{\mathbb{R}^{N}}V^{2^{\ast}-2}_{0}\psi_{0}\varphi dx\Big), \ \ \hbox{as} \ \ n\to \infty.
\end{equation}
In fact, to prove \eqref{Claim1}, we observe that
\begin{equation}\label{PClaim1}
\begin{aligned}
\int_{\mathbb{R}^{N}}&\Big(\frac{1}{|x|^{\mu_n}}\ast(V^{2^{\ast}_{\mu_{n}}-1}_{n}\psi_{n})\Big)V^{2^{\ast}_{\mu_{n}}-1}_{n}\varphi dx\\
&=\int_{\mathbb{R}^{N}}\Big(\frac{1}{|x|^{\mu_n}}\ast(V^{2^{\ast}_{\mu_{n}}-1}_{n}\psi_{n}-V^{2^{\ast}-1}_0\psi_{0})\Big)V^{2^{\ast}_{\mu_{n}}-1}_{n}\varphi dx+\int_{\mathbb{R}^{N}}\Big(\frac{1}{|x|^{\mu_n}}\ast(V^{2^{\ast}-1}_{0}\psi_{0})\Big)V^{2^{\ast}-1}_{0}\varphi dx\\
&\hspace{4mm}+\int_{\mathbb{R}^{N}}\Big(\frac{1}{|x|^{\mu_n}}\ast(V^{2^{\ast}-1}_{0}\psi_{0})\Big)\Big(V^{2^{\ast}_{\mu_{n}}-1}_{n}-V^{2^{\ast}-1}_{0}\Big)\varphi dx.
\end{aligned}
\end{equation}
Consider the first term in \eqref{PClaim1}, by the HLS inequality, we know
\begin{equation}\label{FT1}
\begin{aligned}
\Big|\int_{\mathbb{R}^{N}}&\Big(\frac{1}{|x|^{\mu_n}}\ast(V^{2^{\ast}_{\mu_{n}}-1}_{n}\psi_{n}-V^{2^{\ast}-1}_0\psi_{0})\Big)V^{2^{\ast}_{\mu_{n}}-1}_{n}\varphi dx\Big|\\
&\leq \|V^{2^{\ast}_{\mu_{n}}-1}_{n}\psi_{n}-V^{2^{\ast}-1}_0\psi_{0}\|_{\frac{2N}{2N-\mu}}\|V^{2^{\ast}_{\mu_{n}}-1}_{n}\varphi \|_{\frac{2N}{2N-\mu}}\\
&\leq  \Big(\|(V^{2^{\ast}_{\mu_{n}}-1}_{n}-V^{2^{\ast}-1}_0)\psi_{n}\|_{\frac{2N}{2N-\mu}}+\|V^{2^{\ast}-1}_{0}(\psi_{n}-\psi_{0})\|_{\frac{2N}{2N-\mu}}\Big)\|V^{2^{\ast}_{\mu_{n}}-1}_{n}\varphi \|_{\frac{2N}{2N-\mu}}.
\end{aligned}
\end{equation}
Then, by the expression of $V_n, V_0$, we know that $(V^{2^{\ast}_{\mu_{n}}-1}_{n}-V^{2^{\ast}-1}_0)\in L^{\frac{2N}{N+2-\mu}}(
\mathbb{R}^{N})$, hence
$$
\|(V^{2^{\ast}_{\mu_{n}}-1}_{n}-V^{2^{\ast}-1}_0)\psi_{n}\|_{\frac{2N}{2N-\mu}}\leq \|V^{2^{\ast}_{\mu_{n}}-1}_{n}-V^{2^{\ast}-1}_0\|_{\frac{2N}{N+2-\mu}}\|\psi_{n}\|_{2^{\ast}},
$$
the Dominated Convergence Theorem implies that
$$
\|V^{2^{\ast}_{\mu_{n}}-1}_{n}-V^{2^{\ast}-1}_0\|_{\frac{2N}{N+2-\mu}}\to 0,\ \ \hbox{as} \ \ n\to \infty.
$$
At the mean time, the fact that $\psi_{n}$ converges weakly to $\psi_{0}$ in $D^{1,2}(\mathbb{R}^{N})$ also implies that
$$
\|V^{2^{\ast}-1}_{0}(\psi_{n}-\psi_{0})\|_{\frac{2N}{2N-\mu}}\to 0,\ \ \hbox{as} \ \ n\to \infty.
$$
Similar arguments for the second term in \eqref{PClaim1}, we know
$$
\int_{\mathbb{R}^{N}}\Big(\frac{1}{|x|^{\mu_n}}\ast(V^{2^{\ast}-1}_{0}\psi_{0})\Big)\Big(V^{2^{\ast}_{\mu_{n}}-1}_{n}-V^{2^{\ast}-1}_{0}\Big)\varphi dx\to 0,\ \ \hbox{as} \ \ n\to \infty.
$$
In order to prove
$$
\begin{aligned}
\int_{\mathbb{R}^{N}}\Big(\frac{1}{|x|^{\mu_n}}\ast(V^{2^{\ast}-1}_{0}\psi_{0})\Big)V^{2^{\ast}-1}_{0}\varphi dx\to \Big(\int_{\mathbb{R}^{N}}V^{2^{\ast}-1}_{0}\psi_{0}dx\Big)\Big(\int_{\mathbb{R}^{N}}V^{2^{\ast}-1}_{0}\varphi dx\Big), \ \ \hbox{as} \ \ n\to \infty,
\end{aligned}
$$
 we need to establish a version of convergence property under $L^{\infty}$ norm and a
$L^{\infty}$-estimate for the convolution part for the case $\mu$ is close to $0$.
The proof of $L^{\infty}$-estimate is similar to Lemma 2.6 in \cite{DY} and Proposition 2.5 in \cite{SJ} where the subcritical case was considered.
By H$\ddot{o}$lder's inequality and the boundedness of $V_0$, we know
$$
\begin{aligned}
\Big|\frac{1}{|\cdot|^{\mu_n}}\ast(V^{2^{\ast}-1}_{0}\psi_{0})\Big|&\leq\int_{\mathbb{R}^{N}}\frac{1}{|x-y|^{\mu_n}}|V^{2^{\ast}-1}_{0}\psi_{0}|(y)dy\\
&\leq  \int_{B_1(x)}\frac{1}{|x-y|^{\mu_n}}|V^{2^{\ast}-1}_{0}\psi_{0}(y)|dy+ \int_{B^{c}_1(x)}\frac{1}{|x-y|^{\mu_n}}|V^{2^{\ast}-1}_{0}\psi_{0}(y)|dy\\
&\leq C(\int_{B_1(0)}\frac{1}{|y|^{\mu_n\frac{2^{\ast}}{2^{\ast}-1}}}dy+1) \|\psi_{0}\|_{2^{\ast}}.
\end{aligned}
$$
Thus $\{\frac{1}{|\cdot|^{\mu_n}}\ast(V^{2^{\ast}-1}_{0}\psi_{0})\}$ is uniformly bounded for $\mu$ sufficient close to $0$. Moreover, following the proof of Proposition 2.5 in \cite{SJ}, we also have the following property
$$
\lim_{n\to \infty}\Big|\frac{1}{|\cdot|^{\mu_n}}\ast(V^{2^{\ast}-1}_{0}\psi_{0})-\int_{\mathbb{R}^{N}}V^{2^{\ast}-1}_{0}\psi_{0} dx\Big|_{L^{\infty}(K)}=0
$$
for any compact set $K\subset \R^N$.

Now we may take the limit as $n\rightarrow\infty$ in \eqref{B} and obtain that
\begin{equation}\label{B2}
\begin{aligned}
\int_{\mathbb{R}^{N}}\nabla\psi_0\nabla\varphi dx=2^{\ast}\Big(\int_{\mathbb{R}^{N}}V^{2^{\ast}-1}_{0}\psi_{0}dx\Big)\Big(\int_{\mathbb{R}^{N}}V^{2^{\ast}-1}_{0}\varphi dx\Big)+(2^{\ast}-1)\Big(\int_{\mathbb{R}^{N}}V^{2^{\ast}}_{0}dx\Big)\Big(\int_{\mathbb{R}^{N}}V^{2^{\ast}-2}_{0}\psi_{0}\varphi dx\Big),
\end{aligned}
\end{equation}
 for any $\varphi\in D^{1,2}(\mathbb{R}^{N})$. Hence $\psi_{0}$ satisfies
\begin{equation}\label{C}
-\Delta \psi_{0} ={2}^{\ast}\int_{\mathbb{R}^N}V_0^{{2}^{\ast}-1}\psi_{0} dx V_0^{{2}^{\ast}-1} +({2}^{\ast}-1)\int_{\mathbb{R}^N} V_0^{2^\ast}dxV_0^{{2}^{\ast}-2}\psi_{0}.
\end{equation}
By the nondegeneracy of $V_{0}$, we know
\begin{equation}\label{NR}
\psi_{0}\in\mathbf{span}\Big\{\partial_{1}V_{0},\partial_{2}V_{0},\cdots\partial_{N}V_{0},\partial_{t}V_{0}\Big\}.
\end{equation}

We prove that $\psi_{0}\neq0$. We take $\varphi_{n}=\psi_{n}$ in equation \eqref{B} to get
\begin{equation}\label{NTrival}
\begin{aligned}
\int_{\mathbb{R}^{N}}|\nabla\psi_n|^2 dx=&2^{\ast}_{\mu_n}\int_{\mathbb{R}^{N}}\Big(\frac{1}{|x|^{\mu_n}}\ast(V^{2^{\ast}_{\mu_{n}}-1}_{n}\psi_n)\Big)V^{2^{\ast}_{\mu_{n}}-1}_{n}\psi_n dx\\
&+(2^{\ast}_{\mu_n}-1)\int_{\mathbb{R}^{N}}\Big(\frac{1}{|x|^{\mu_n}}\ast V^{2^{\ast}_{\mu_{n}}}_n\Big)V^{2^{\ast}_{\mu_{n}}-2}_{n}\psi^2_n dx.
\end{aligned}
\end{equation}
However, on one hand,
$$
\int_{\mathbb{R}^{N}}|\nabla\psi_n|^2 dx=1.
$$
On the other hand, keep Proposition \ref{Convergence} in mind, we can repeat the arguments in Claims \eqref{Claim1} and \eqref{Claim2} to show
$$
\begin{aligned}
2^{\ast}_{\mu_n}&\int_{\mathbb{R}^{N}}\Big(\frac{1}{|x|^{\mu_n}}\ast(V^{2^{\ast}_{\mu_{n}}-1}_{n}\psi_n)\Big)V^{2^{\ast}_{\mu_{n}}-1}_{n}\psi_n dx+(2^{\ast}_{\mu_n}-1)\int_{\mathbb{R}^{N}}\Big(\frac{1}{|x|^{\mu_n}}\ast V^{2^{\ast}_{\mu_{n}}}_n\Big)V^{2^{\ast}_{\mu_{n}}-2}_{n}\psi^2_n dx\\
&\hspace{4mm}\longrightarrow2^{\ast}\Big(\int_{\mathbb{R}^{N}}V^{2^{\ast}-1}_{0}\psi_{0}dx\Big)^2+(2^{\ast}-1)\Big(\int_{\mathbb{R}^{N}}V^{2^{\ast}}_{0}dx\Big)\Big(\int_{\mathbb{R}^{N}}V^{2^{\ast}-2}_{0}|\psi_{0}|^2 dx\Big),
\end{aligned}
$$
therefore $\psi_{0}\neq0$.

Since we know
$$
\psi_{n}\in T_{z}{Z_{\mu}}^\bot=\mathbf{span}\Big\{\partial_{1}V_0,\partial_{2}V_0,\cdots,\partial_{N}V_0,\partial_{t}V_0\Big\}^\bot,
$$
then for every
$$
\eta_0=aD_{t}V_{0}+\mathbf{b}\cdot\nabla V_{0}\in \mathbf{span}\Big\{\partial_{1}V_0,\partial_{2}V_0,\cdots,\partial_{N}V_0,\partial_{t}V_0\Big\},
$$
where $a\in\mathbb{R},\mathbf{b}=(b_{1},b_{2},\cdots,b_{N})\in\mathbb{R}^{N}$, we have
$$
<\psi_{n},\eta_{0}>=0
$$
Where we denote $<\cdot,\cdot>$ as the inner product in $D^{1,2}(\mathbb{R}^{N})$. However, as $n\rightarrow\infty$, we know
$$
<\psi_{0},\eta_{0}>=0.
$$
This contradicts to  \eqref{NR} that $\psi_{0}\in\mathbf{span}\Big\{\partial_{1}V_{0},\partial_{2}V_{0},\cdots,\partial_{N}V_{0},\partial_{t}V_{0}\Big\}$, since we proved that $\psi_{0}\neq0$. Hence any solution satisfies \eqref{A} must belong to $T_{z}Z$ in the space $L^{2^{\ast}}$, that is $T_{z}{Z_{\mu}}=\mathbf{Ker}{[J_\mu^{\prime\prime}(z)]}$, we finish the proof.

$\hfill{} \Box$


\end{document}